\documentclass[11pt,twoside]{amsart}

\usepackage[english]{babel}

\usepackage{amssymb, newpxtext, xcolor}
\usepackage[all,cmtip]{xy}
\usepackage[colorinlistoftodos]{todonotes}
\usepackage{graphicx}
\usepackage[colorlinks=true, allcolors=blue]{hyperref}

\newtheorem{dfn}{Definition}[section]
\newtheorem{thm}[dfn]{Theorem}
\newtheorem{crl}[dfn]{Corollary}
\newtheorem{rem}[dfn]{Remark}
\newtheorem{example}[dfn]{Example}

\newenvironment{preuve}{{\em \bf Proof:}}{\hfill $\blacksquare$}

\def\P{\mathbb P}

\def\Q{\mathbb Q}
\def\R{\mathbb R}
\def\Z{\mathbb Z}
\def\H{\mathbb H}

\def\C{\mathbb C}
\def\T{\mathbb T}

\definecolor{green}{rgb}{0,.6,0}

\title[Nonrationality Degree]{Nonrationality Degree of Toric Quasifolds}
\author[Fiammetta Battaglia and Elisa Prato]{Fiammetta Battaglia and Elisa Prato}

\usepackage{amsmath}
\begin{document}
\begin{abstract} 
Toric quasifolds are nonrational generalizations of toric manifolds and orbifolds. We introduce the notion of \emph{nonrationality degree}, an invariant that measures the extent to which a toric quasifold fails to be Hausdorff. We do so by first defining analogous notions for the underlying quasilattice and the corresponding quasitorus. We conclude by applying the nonrationality degree to reframe Gordan's lemma in the nonrational setting.
\end{abstract}
\maketitle
\begin{small}
\noindent \textbf{Keywords.} Toric quasifold, quasilattice, quasitorus, nonrational fan, nonrational polytope.
\smallskip

\noindent \textbf{Mathematics~Subject~Classification:}
Primary: 14M25. Secondary: 52B20.
\end{small}

\section*{Introduction}\label{introduzione}
Toric quasifolds are a natural generalization of toric manifolds and orbifolds in the nonrational case. They were first introduced 
by the second author, in the symplectic setting, in \cite{p}. Their complex and algebraic--geometric counterparts were introduced by both authors in \cite{cx} and \cite{atq}, respectively.
These spaces are locally modeled by $\C^n\cong\R^{2n}$ modulo the action of countable groups
and are generally not Hausdorff.

The aim of this article is
to introduce the notion of \textit{nonrationality degree} of toric quasifolds, 
which measures the extent to which they fail to be Hausdorff.

We recall that 
the starting data for a toric quasifold is provided by the 
\textit{fundamental triple}
$$(\Sigma, Q, \{v_1,\ldots,v_d\}),$$
where $\Sigma$ is a complete simplicial fan in $\R^n$, $Q$ is a quasilattice, namely the $\Z$--span of a set of $\R$--spanning vectors in $\R^n$, and the vectors $v_1,\ldots,v_d$ are in $Q$ and generate the rays of $\Sigma$. The corresponding toric quasifold, which we denote by $X_\Sigma$, is acted upon by the real quasitorus $\R^n/Q$ in the symplectic setting, and by the complex quasitorus $\C^n/Q$ in the complex and algebraic--geometric settings.

In order to define the nonrationality degree of $X_\Sigma$, we first investigate the structure of the underlying quasilattice $Q$. 

As a consequence of the structure theorem for closed groups of $\R^n$, we prove that $Q$ admits a canonical decomposition
$$
Q=(Q\cap V)\oplus(Q\cap W),
$$
where $V$ and $W$ are complementary subspaces of $\R^n$,
$Q\cap V$ is dense in $V$ and $Q\cap W$ is a lattice in $W$ (see Theorem~\ref{quasi decomposition}).
The subspace $V$ is unique and maximal with this property. We define the 
nonrationality degree of $Q$ to be the dimension of $V$. This integer will ultimately define the nonrationality degree of the real and complex quasitori and of the toric quasifold.

In fact, the decomposition of $Q$ naturally yields structure theorems for the corresponding quasitori. For example, in the complex case, we get
$$\C^n/Q\cong D_{\C}^h\times (\C^*)^{n-h},$$ where $h$ is the nonrationality degree of $Q$ and 
$D_{\C}^h$ is a totally nonrational complex quasitorus of dimension $h$ (see Theorem~\ref{complexquasitorusdecomp}). Accordingly, we define
$h$ to be the nonrationality degree of $\C^n/Q$. The same applies to $\R^n/Q$.
Observe that $Q$ is a lattice and both quasitori are tori if and only if their nonrationality degree $h$ is zero. Increasing values of $h$ correspond to the presence of higher--dimensional totally nonrational directions.

As in the rational case, the action of the complex quasitorus has a dense open orbit. This naturally leads to defining the notion of nonrationality degree of the toric quasifold to coincide with that of its residing quasitorus.

A given complete simplicial fan may be viewed in the context of infinitely many different triples, and
the degree of the corresponding quasifolds will vary accordingly. We take the minimum possible nonrationality degree of these quasifolds to define the intrinsic nonrationality degree
of the fan. The fan is rational if and only if its nonrationality degree is zero. This applies also to simple polytopes, by considering their normal fan. This invariant is purely convex--geometric. It has the interesting feature of providing a lower bound for the nonrationality degree of the corresponding toric quasifolds.

The motivation for introducing these notions of nonrationality degree is twofold. On the one hand, it comes from the framework of LVMB manifolds, a special class of complex foliated manifolds \cite{ldm,M,bosio}, and their measures of nonrationality (see \cite[Section~2.2.3]{bz1}). The relation between these measures and the nonrationality degree of quasifolds is currently being investigated by F.~Thiella. On the other hand, the motivation also stems from the study of nonrational toric geometry in the algebraic--geometric setting (see \cite{atq}). 
In this framework, replacing a lattice with a quasilattice has significant repercussions, such as  the different behavior of duals (see Remarks~\ref{hom} and \ref{ring}) and the failure of Gordan's Lemma, a fundamental result in classical toric geometry. In fact, we recast the lemma by proving that it holds if and only if the nonrationality degree of the quasilattice is zero (see Theorem~\ref{gordanlemma}).
In other words, any positive nonrationality degree precludes the classical finite generation property.

We conclude by remarking that quasilattices are key objects in several different areas of mathematics. For example, they form the basis for the study of nonperiodic tilings that are related to quasicrystals (see, for example, \cite{senechal, kite}). They also appear, explicitly or implicitly, in most other works addressing nonrational toric geometry (see \cite{bz1, ratiu, HS, H, KLMV, IKP, boivin}). Moreover, they arise naturally in the framework of diffeology (see, for example, \cite{di, IZP}). 

We expect these notions of nonrationality degree to provide a conceptual basis for future developments in nonrational toric and convex geometry, as well as in other mathematical settings in which quasilattices arise naturally.

The paper is organized as follows. In Section~\ref{quasilattice}, we investigate the structure of quasilattices and define their nonrationality degree. In Section~\ref{quasitorus}, we establish the structure theorems for quasitori and introduce their nonrationality degree. In Section~\ref{quasifold}, we extend this notion to toric quasifolds. In Section~\ref{fan}, we discuss the case of fans and polytopes. Finally, in Section~\ref{gordan}, we reframe Gordan's lemma in the nonrational setting.

\section{Nonrationality degree of a quasilattice}\label{quasilattice}

We begin by recalling the following
\begin{dfn}\rm{A \textit{quasilattice} $Q$ in $\R^n$ is the $\Z$--span of a set of $\R$--spanning vectors of $\R^n$.}
\end{dfn}
A quasilattice is, in particular, a free $\Z$--module of finite rank and a subgroup of $\R^n$. The 
\textit{rank} of $Q$ is the number of vectors in a set of minimal generators.
We recall that a lattice $L$ in $\R^n$ is the $\Z$--span of a basis of $\R^n$. It is therefore a quasilattice of rank equal to $n$. On the other hand, the classical definition of lattice \cite[Chapter~7, Section~2.2]{serre} is that of a subgroup $L$ of $\R^n$ that satisfies the following equivalent conditions:
\begin{itemize}
\item $L$ is discrete and $\R^n/L$ is compact;
\item $L$ is discrete and spans $\R^n$;
\item there exists an $\R$--basis of $\R^n$ which is a $\Z$--basis of $L$.
\end{itemize}
Therefore, a quasilattice $Q$ in $\R^n$ is a lattice if and only if it is a discrete subset of $\R^n$. 

The aim of this section is to introduce the notion of \textit{nonrationality degree} of a quasilattice. 
We will be making use of the classification theorem of closed subgroups of $\R^n$, which we recall from Bourbaki \cite[Chapter VII, Section 1.2, Thm. 2]{bourbaki}:
\begin{thm}\label{closed}
Let $G$ be a closed subgroup of $\R^n$, and let $k=\dim(\operatorname{span}_{\R}(G))$. Then there exists a maximal vector subspace $V\subset \R^n$, $\dim(V)\leq k$, such that, for every complement $W$ of $V$, $\Lambda=W\cap G$ is a discrete subgroup of rank $r=k-\dim V$ and $G=V\oplus\Lambda$.
\end{thm}
Given a quasilattice $Q$ in $\R^n$, we apply this classification theorem to the closed subgroup $\overline{Q}$ and obtain the following
\begin{thm}\label{quasi decomposition}
Let $Q$ be a quasilattice of $\R^n$. Then there exists a unique maximal vector subspace
$V\subset \R^n$ and a complement $W$ of $V$ such that
$$
Q=(Q\cap V)\oplus(Q\cap W),
$$
where $Q\cap V$ is dense in $V$ and $Q\cap W$ is a lattice in $W$.
\end{thm}
\noindent\begin{preuve}
Since $Q$ is a subgroup of $\R^n$, $\overline{Q}$ is a closed subgroup of $\R^n$. Moreover, since
$Q$ spans $\R^n$, so does $\overline{Q}$.
Applying Theorem~\ref{closed} to $\overline{Q}$, we obtain a
maximal vector subspace $V\subset\R^n$ such that, for every complement $W$ of
$V$, we have
$\overline{Q}=V\oplus\Lambda$,
where $\Lambda=\overline{Q}\cap W$
is a 
discrete subgroup of rank
$r=n-\dim(V)$, thus a lattice in $W$.
Let
$$
\pi:\R^n\longrightarrow \R^n/V
$$
be the quotient map. Since $Q$ is dense in $\overline{Q}$ and $\pi$ is continuous,
$\pi(Q)$ is dense in $\pi(\overline{Q})$. Since
$$
\pi(\overline{Q})\cong \overline{Q}/V\cong\Lambda,
$$
the group $\pi(\overline{Q})$ is a lattice in $\R^n/V$; in particular, it is discrete.

Let us show that
$\overline{\pi(Q)}=\pi(\overline{Q})$.
Since $Q\subset \overline{Q}$, we have
$
\overline{\pi(Q)}\subset \overline{\pi(\overline{Q})}.
$
However, $\pi(\overline{Q})$ is discrete, hence closed. Therefore
$
\overline{\pi(Q)}\subset \pi(\overline{Q}).
$
Conversely, consider an element $\pi(h)$, $h\in \overline{Q}$ and take
a sequence $(q_m)\subset Q$ converging to $h$.
By continuity, the sequence $(\pi(q_m))\subset \pi(Q)$ converges to $\pi(h)$.
Therefore
$\pi(\overline{Q})\subset \overline{\pi(Q)}$.

Since $\pi(\overline{Q})$ is discrete and $\pi(Q)$ is dense in
$\pi(\overline{Q})$, we obtain
$
\pi(Q)=\pi(\overline{Q}).
$

Now choose a $\Z$--basis of $\pi(Q)$
and write it as
$\{\pi(w_1),\ldots,\pi(w_r)\}$,
with $w_1,\ldots, w_r \in Q$. Take the subspace
$$
W=\operatorname{span}_{\R}(w_1,\ldots,w_r).
$$
Since the vectors $\pi(w_1),\ldots,\pi(w_r)$ form a basis of
$\R^n/V$, the subspace $W$ is a complement of $V$.

We claim that
$$
Q\cap W
=
\operatorname{span}_{\Z}(w_1,\ldots,w_r).
$$
Let $w\in Q\cap W$. Since $w\in W$, there exist real numbers
$a_1,\ldots,a_r$ such that
$w=\sum_{i=1}^r a_iw_i.$
Thus $\pi(w)=\sum_{i=1}^r a_i\pi(w_i)$.
Since $w\in Q$, its image via $\pi$ belongs to $\pi(Q)$, and therefore
$\pi(w)=\sum_{i=1}^r m_i\,\pi(w_i)$,
for uniquely determined integers $m_1,\ldots, m_r$. Thus $a_i=m_i$
for all $i$. Therefore
$w\in \operatorname{span}_{\Z}(w_1,\ldots,w_r).$
The reverse inclusion is obvious, proving the claim.
It follows that $Q\cap W$ is a discrete subgroup of $W$ of  rank
$r=n-\dim(V)$, and is thus a lattice in $W$.

Now let $u\in Q$. Since the classes $\pi(w_1),\ldots,\pi(w_r)$
generate $\pi(Q)$, there exist integers $h_1,\ldots,h_r$ such that
$\pi(u)=\sum_{i=1}^r h_i\,\pi(w_i)$.
Therefore
$$
u-\sum_{i=1}^r h_iw_i\in V.
$$
Since both terms belong to $Q$, we obtain
$$
u-\sum_{i=1}^r h_iw_i\in Q\cap V.
$$
Thus every element of $Q$ can be written as the sum of an element of
$Q\cap V$ and an element of $Q\cap W$. Therefore
$Q=(Q\cap V)+(Q\cap W)$.
Since $V\cap W=\{0\}$, we have
$(Q\cap V)\cap(Q\cap W)=\{0\}$,
and therefore
$$
Q=(Q\cap V)\oplus(Q\cap W).
$$

Let us now show that $Q\cap V$ is dense in $V$.
Let $v\in V\subset \overline{Q}$, and take a sequence $(q_m)\subset Q$ converging to $v$. Since $Q=(Q\cap V)\oplus (Q\cap W)$,  we can write $q_m=v_m+w_m$, with $v_m\in Q\cap V$ and $w_m\in Q\cap W$. Now, since $\R^n=V\oplus W$, we can consider the linear projection $p_V\colon \R^n\rightarrow V$. Since it is continuous, the sequence $(p_V(q_m))=(v_m)$ converges to $p_V(v)=v$. Therefore $Q\cap V$ is dense in $V$.

To conclude, we show that $V$ is unique. Suppose there are two maximal subspaces $V_1$ and $V_2$ such that $Q\cap V_1$ is dense in $V_1$, $Q\cap V_2$ is dense in $V_2$. Notice that $Q\cap (V_1+V_2)$ is dense in $V_1+ V_2$. Indeed, take $v_1+v_2\in V_1+V_2$. We know that there exists a sequence
$(v^1_m)\subset Q\cap V_1$ converging to $v_1$ and a sequence
$(v^2_m)\subset Q\cap V_2$ converging to $v_2$. Then the sequence $(v^1_m+v^2_m)\subset Q\cap (V_1+V_2)$ converges to $v_1+v_2$. Therefore, by maximality, $\dim(V_1+V_2)\leq \dim(V_1)=\dim(V_2)$, which implies $V_1=V_2$.   
\end{preuve}
\begin{dfn}[Maximal totally nonrational subspace]
\rm{Given a quasilattice $Q$ in $\R^n$, we call the unique subspace $V$ of Theorem~\ref{quasi decomposition} \textit{maximal totally nonrational subspace} with respect to $Q$.}
\end{dfn}

We are now finally ready for the following
\begin{dfn}[Nonrationality degree of a quasilattice]
\rm{
Let $Q$ be a quasilattice in $\R^n$. The \textit{nonrationality degree} of $Q$, denoted
 $\deg_{\textsc{NR}}(Q)$, is the dimension of the maximal totally nonrational subspace $V$.}
\end{dfn}
We clearly have $$0\leq \deg_{\textsc{NR}}(Q)\leq n.$$

As a first, straightforward, consequence of
Theorem~\ref{quasi decomposition}, we get
\begin{crl}\label{grado quasireticolo} Let $Q$ be a quasilattice in $\R^n$. Then $\deg_{\textsc{NR}}(Q)=0$ if and only if $Q$ is a lattice.
\end{crl}

One of the crucial differences between lattices and quasilattices is in the notion of dual. Let us first recall from \cite{atq} the following
\begin{dfn}[Weak dual of a quasilattice]
\rm{Given a quasilattice $Q$ in $\R^n$, we define
the \textit{weak dual} of $Q$ to be the $\Z$--module
$$Q^\vee=\{\lambda\in(\R^n)^*\mid\lambda(Q)\subset \Z\}.$$}
\end{dfn}
As a consequence of Theorem~\ref{quasi decomposition}, we see that the elements of 
$Q^\vee$ only keep track of rational directions:
\begin{crl}\label{weakdual}
Let $Q$ be a quasilattice of $\R^n$ and consider the subspaces $V$, $W$ of Theorem~\ref{quasi decomposition}. Then we have
\begin{enumerate}
\item $Q^\vee\cong(Q\cap W)^*$ and is thus a discrete subgroup
of $(\R^n)^*$ of rank $r=n-\deg_{\textsc{NR}}(Q)$;
\item 
$V=\operatorname{Ann}(Q^\vee)$.
\end{enumerate}
\end{crl}
\noindent\begin{preuve}
(1)
The direct sum decomposition
$\R^n=V\oplus W$ allows us to 
write $\lambda=\lambda|_V+\lambda|_W$
for every $\lambda\in (\R^n)^*$.  Now notice that, 
if $\lambda \in Q^\vee$,
by continuity, $\lambda(\overline{Q})\subset \Z$. Thus, since $V\subset \overline{Q}$
and $V$ is connected, we have $\lambda|_V=0$.
Therefore, $\lambda=\lambda|_W$ and
$$Q^\vee=\{\lambda\in(\R^n)^*\mid\lambda(Q)\subset \Z\}\cong\{\mu\in  W^*\mid\mu(Q\cap W)\subset \Z\}=(Q\cap W)^*.$$

\noindent (2) We have already noticed that $\lambda|_V=0$, for all $\lambda\in Q^\vee$.
Thus $V\subset \operatorname{Ann}(Q^\vee)$.
However, by (1) we have
$$\dim(\operatorname{Ann}(Q^\vee))=n-\operatorname{rank}(Q^\vee)=n-(n-\deg_{\textsc{NR}}(Q))=\dim(V).$$ So the two subspaces must coincide.
\end{preuve}

\begin{rem}\label{hom}
\rm{
Consider now the actual dual,
$\operatorname{Hom}(Q,\Z)$; it is a free $\Z$--module having the same rank as $Q$. 
When $Q$ is a lattice, $\operatorname{Hom}(Q,\Z)$ is naturally identified with $Q^\vee$ and
can be thus thought of as a lattice in $(\R^n)^*$. 
However, if $\deg_{\textsc{NR}}(Q)\neq 0$, by Corollary~\ref{weakdual},
this is no longer the case. This fact has 
some notable consequences in nonrational toric geometry (see \cite{atq}
and Remark~\ref{ring} below).}
\end{rem}
It will be useful, for what is to follow, to have Theorem~\ref{quasi decomposition} reframed as
\begin{crl}\label{iso}
Let $Q$ be a quasilattice in $\R^n$ and let $h=\deg_{\textsc{NR}}(Q).$ Then
there is an isomorphism
$$f\colon\R^n\longrightarrow \R^n=\R^h\oplus\R^{n-h}$$
such that 
$$Q\cong f(Q)=(f(Q)\cap \R^h)\oplus\Z^{n-h},
$$
where $f(Q)\cap \R^h$ is a dense quasilattice in $\R^h$.
\end{crl}
\noindent\begin{preuve}
Consider the subspaces $V$ and $W$ of Theorem~\ref{quasi decomposition}.
Choose a basis $\{v_1,\ldots ,v_h, w_1,\ldots, w_{n-h}\}$ of $\R^n$ such that $\{v_1,\ldots,v_h\}$ is a basis of $V$ and
$\{w_1,\ldots,w_{n-h}\}$ is a $\Z$--basis of $Q\cap W$. Let $f$ be the isomorphism that sends
this basis to the standard basis of $\R^n$. 
\end{preuve}

We conclude this section with some examples.
\begin{example}\label{quasiuno}
\rm{Let us consider the quasilattice in $\R$ given by $Q=\Z+a\Z$, with $a$ a 
\textit{positive} real number. 
The nonrationality degree obviously depends on $a$. In fact,
if $a\in\Z$, we have $Q=\overline{Q}=\Z,\; V=\{0\},\;  Q^\vee=\Z.$ If $a\in \Q \setminus \Z$, and we write $a=\frac pq$, with $p,q \in \Z_{> 0}$ coprime, then 
$Q=\overline{Q}=\frac{1}{q}\Z,\; V=\{0\},\; Q^{\vee}=q\Z.$
Finally, if $a\notin \Q$, then
$\overline{Q}=\R,\; V=\R,\; Q^{\vee}=\{0\}.$
In conclusion,
$$\deg_{\textsc{NR}}(Q)=
\left\{\begin{array}{l}
0 \quad\quad a\in\Q \\
1 \quad\quad a\notin\Q.
\end{array}
\right.
$$}
\end{example}
\begin{example}\label{quasidue}
\rm{Consider the quasilattice $Q=\Z\times(\Z+a\Z)$ in $\R^2$, where
$a$ is a \textit{positive} real number.
Then, if $a\in\Z$, we have $Q=\overline{Q}=\Z^2,\, V=\{0\},\; Q^\vee=\Z^2.$ If $a\in \Q\setminus \Z$,
and we write $a=\frac pq$, with $p,q \in \Z_{> 0}$ coprime, then 
$Q=\overline{Q}=\Z\times\frac{1}{q}\Z,\; V=\{0\},\;Q^{\vee}=\Z\times q\Z.$
Finally, if $a\notin \Q$,
then $\overline{Q}=\Z\times \R,\; V=\{0\}\times\R, \; Q^{\vee}=\Z\times \{0\}.$
Here we have, again,
$$\deg_{\textsc{NR}}(Q)=
\left\{\begin{array}{l}
0 \quad\quad a\in\Q \\
1 \quad\quad a\notin\Q.
\end{array}
\right.
$$}
\end{example}
\begin{example}\label{quasicinque}
\rm{Consider the fifth roots of unity in $\C\cong\R^2$
$$\begin{array}{l}
u_0=(1,0)\\ 
u_1=(\cos{\frac{2\pi}{5}},\sin{\frac{2\pi}{5}})\\ u_2=(\cos{\frac{4\pi}{5}},\sin{\frac{4\pi}{5}})\\ u_3=(\cos{\frac{6\pi}{5}},\sin{\frac{6\pi}{5}})\\ u_4=(\cos{\frac{8\pi}{5}},\sin{\frac{8\pi}{5}}).\end{array}$$
Their $\Z$--span, $Q_5$, is a dense quasilattice in $\R^2$. 
To see this, notice that the fifth roots of unity generate a $4$--dimensional $\mathbb{Q}$--vector space, since the fifth cyclotomic polynomial has degree $4$. Hence $\text{rank}(Q_5)=4$. By Theorem~\ref{quasi decomposition} and Corollary~\ref{grado quasireticolo}, there exists a unique maximal linear subspace $V$ of $\R^2$ such that $Q_5\cap V$ is dense in $V$, and $\text{dim}(V)\geq 1$. Since $Q_5$ is invariant under rotation by an angle $2\pi/5$, by uniqueness, so is $V$. As no real one--dimensional subspace is invariant under such a rotation, it follows that $\dim V=2$, hence $V=\mathbb{R}^2$.
This quasilattice turns out to be the crucial ingredient in the study of Penrose rhombus
tilings and of the Penrose kite from the nonrational symplectic toric viewpoint
(see \cite{kite} and references therein, and also Example~\ref{kite} below).
Here we have
$\overline{Q}=V=\R^2, \; Q^{\vee}=0$ and, finally, $\deg_{\textsc{NR}}(Q)=2$.
}
\end{example}
\section{Structure theorems for quasitori}\label{quasitorus}
We begin by recalling the following
\begin{dfn}[Real and complex quasitori]{\rm Let $Q$ be a quasilattice in $\R^n$, the group $\R^n/Q$ is called a 
\textit{real quasitorus} of dimension $n$ and the group $\C^n/Q$ is called a \textit{complex quasitorus} of dimension $n$. 
}    
\end{dfn}
Notice that when $Q$ is a lattice, these groups are, respectively, a real torus and a complex torus. At the other extreme we have
\begin{dfn}[Totally nonrational quasitori]
\rm{Consider a quasilattice $Q$ in $\R^n$. We say that
the quasitori $\R^n/Q$ and $\C^n/Q$ are \textit{totally nonrational}
if $Q$ is dense in $\R^n$.}
\end{dfn}
The general case is a bit of both, as described in the following structure theorems.
\begin{thm}[Structure theorem for a real quasitorus]\label{realquasitorusdecomp}
Let $Q$ be a quasilattice in $\R^n$, and
let $h=\deg_{\textsc{NR}}(Q)$. Then
$$\R^n/Q\cong D^h\times \T^{n-h},$$ where
$D^h$ is a totally nonrational real quasitorus of dimension $h$ and $\T^{n-h}$ is the standard torus of dimension $n-h$.
\end{thm}
\noindent\begin{preuve}
By Corollary~\ref{iso}, we have
$$Q\cong(f(Q)\cap \R^h)\oplus\Z^{n-h},$$ where $f(Q)\cap \R^h$ is dense in $\R^h$. Thus
$$\R^n/Q\cong \left(\R^{h}/(f(Q)\cap\R^h)\right) \times \left(\R^{n-h}/\Z^{n-h}\right).$$
\end{preuve}

Similarly, in the complex case, we have
\begin{thm}[Structure theorem for a complex quasitorus]\label{complexquasitorusdecomp}
Let $Q$ be a quasilattice in $\R^n$, and
let $h=\deg_{\textsc{NR}}(Q)$. Then
$$\C^n/Q\cong D_{\C}^h\times (\C^*)^{n-h},$$ where
$D_{\C}^h$ is a totally nonrational complex quasitorus of dimension $h$.
\end{thm}
Motivated by this, we have the following
\begin{dfn}[Nonrationality degree of a quasitorus]
\rm{
Let $Q$ be a quasilattice in $\R^n$. We call \textit{nonrationality degree} of the quasitorus $\R^n/Q$ (respectively $\C^n/Q$), denoted $\deg_{\textsc{NR}}(\R^n/Q)$ (respectively $\deg_{\textsc{NR}}(\C^n/Q)$), the nonrationality degree of $Q$.}
\end{dfn}
Let us see some examples. To avoid unnecessary repetition, we will be focusing on the real case.

\begin{example}\label{quasitorouno}
\rm{Let us consider the quasilattice $\Z+a\Z$ in $\R$ , where $a$ is a \textit{positive} real number
(see Example~\ref{quasiuno})
and take the quasitorus $$\R/(\Z+a\Z).$$
If $a\in\Z$, we obtain $\R/\Z\cong S^1$.
If $a\in \Q\setminus \Z$, and we write $a=\frac pq$, with $p,q \in \Z_{> 0}$ coprime, we get
$\R/\left(\frac{1}{q}\Z\right)\cong S^1$. 
Finally, if $a\notin \Q$, we get the \textit{irrational torus} $\mathbf{T}^a$, 
which was first introduced within the framework of diffeology by Donato--Iglesias \cite{di}.
In conclusion, we have
$$\deg_{\textsc{NR}}(\R/(\Z+a\Z))=
\left\{\begin{array}{l}
0 \quad\quad a\in\Q \\
1 \quad\quad a\notin\Q.
\end{array}
\right.
$$}
\end{example}
\begin{example}\label{quasitorodue}
\rm{Consider the quasilattice $Q=\Z\times(\Z+a\Z)$ in $\R^2$, where
$a$ is a \textit{positive} real number (see Example~\ref{quasidue}) and take the quasitorus
$$\R^2/(\Z\times(\Z+a\Z)).$$
Then, if $a\in\Z$, we obtain 
$\R^2/\Z^2$.
If $a\in \Q\setminus \Z$,
and we write $a=\frac pq$, with $p,q \in \Z_{> 0}$ coprime, then we get
$(\R/\Z)\times \left(\R/\frac{1}{q}\Z\right)\cong \R^2/\Z^2$.
Finally, if $a\notin \Q$ we get $(\R/\Z)\times \left(\R/(\Z+a\Z)\right)\cong S^1\times \mathbf{T}^a$.
In particular,
$$\deg_{\textsc{NR}}(\R^2/(\Z\times(\Z+a\Z)))=
\left\{\begin{array}{l}
0 \quad\quad a\in\Q \\
1 \quad\quad a\notin\Q.
\end{array}
\right.
$$

}
\end{example}
\begin{example}\label{quasitorocinque}
\rm{Take the  $\Z$--span, $Q_5$, of the fifth roots of unity (see Example~\ref{quasicinque}).
Since $Q_5$ is a dense quasilattice in $\R^2$, the quasitorus $\R^2/Q_5$ is totally nonrational
and its nonrationality degree is $2$.
}
\end{example}

\section{Nonrationality degree of a toric quasifold}\label{quasifold}

Let $\Sigma$ be a complete simplicial fan in $\R^n$ (we refer the reader to \cite{cox} for the necessary background on fans).

We recall that a \textit{fundamental triple} for $\Sigma$ is given by
$$(\Sigma, Q, \{v_1,\ldots,v_d\}),$$ where 
$Q$ is a quasilattice in $\R^n$, and the vectors $v_1,\ldots,v_d\in Q$ generate the rays of $\Sigma$.
One way of obtaining a fundamental triple for \textit{any} $\Sigma$ is to first choose a set of ray generators and then take $Q$ to be their $\Z$--span.

We recall that the fan $\Sigma$ is \textit{rational} if all of its rays have nonzero intersection with a lattice $L$. Notice that, in this case, there is an underlying canonical triple, obtained by taking $Q$ to be equal to $L$ and the ray generators to be primitive in $L$.

To any fundamental triple, there corresponds a toric quasifold, which we will denote
by $X_\Sigma$. For the symplectic construction of $X_\Sigma$, which holds when $\Sigma$ is also polytopal, see \cite{p}; its complex and algebraic--geometric counterparts are treated in \cite{cx} and \cite{atq}, respectively. The space $X_\Sigma$ is a toric manifold or orbifold if and only if the fan $\Sigma$ is rational and $Q$ is a lattice, if and only if $\deg_{\textsc{NR}}(Q)=0$.

Analogously to what happens in the rational case, the quasitorus acts on $X_\Sigma$ and, in the complex setting, there is a dense open orbit $$\mathcal{O}_o\cong \C^n/Q$$ which corresponds to the zero cone; all other points of $X_\Sigma$ belong to the lower dimensional orbits corresponding to the nonzero cones of $\Sigma$. It is natural to define the nonrationality degree of $X_\Sigma$ to be the same as that of the quasitorus, namely

\begin{dfn}[Nonrationality degree of a toric quasifold]
\rm{We define the \textit{nonrationality degree} of $X_\Sigma$ to be the nonrationality degree of $Q$, denoted $\deg_{\textsc{NR}}(X_\Sigma)$.}
\end{dfn}

\begin{dfn}[Totally nonrational quasifold]
\rm{We say that the quasifold $X_\Sigma$ is \textit{totally nonrational} when the
corresponding quasilattice $Q$ is dense.}
\end{dfn}

Let us compute the nonrationality degree of some notable examples of toric quasifolds.
\begin{example}[Quasisphere]\label{quasisphere}
\rm{Consider the \textit{quasisphere} $X_a$, namely the
toric quasifold corresponding to the triple
$$
(\Sigma, \Z+a\Z, \{a,-1\}),
$$
where $a$ is a positive real number. $\Sigma$ is the complete simplicial fan in $\R$ given 
by the cones $\R_{\geq 0}$, $\R_{\leq 0}$, and $\{0\}$. We refer the
reader to \cite[Example~1.13]{p} for the symplectic construction, to \cite[Examples~2.6, 3.8]{cx}, \cite[Example~3.1]{laurent} for the complex, and 
to \cite{atq} for the algebraic--geometric. 

Since the underlying quasilattice is $\Z+a\Z$, we have, by Example~\ref{quasiuno},
$$\deg_{\textsc{NR}}(X_a)=
\left\{\begin{array}{l}
0 \quad\quad a\in\Q \\
1 \quad\quad a\notin\Q.
\end{array}
\right.
$$
The quasisphere is totally nonrational if and only if $a\notin\Q$.}
\end{example}
\begin{example}[Generalized weighted projective space]\label{quasiprojective}
\rm{We consider now \textit{generalized weighted projective space} $\C\P^2_{(1,1,a)}$.
It is the toric quasifold corresponding to the triple
$$
(\Sigma, \mathbb{Z}\times (\mathbb{Z}+a\mathbb{Z}), \{(1,0), (-1,-a), (0,1)\}),
$$
where $a$ is any positive real number. $\Sigma$ 
is the complete simplicial fan in $\R^{2}$ whose rays are generated by the vectors  
 $(1,0)$, $(-1,-a)$, and $(0,1)$. See \cite[Example~3.6]{p}, \cite[Example~2.8]{cx}, \cite[Example~3.2]{laurent}, and \cite{atq} for descriptions of this toric quasifold in the symplectic, complex, and algebraic--geometric settings, respectively.

Since the underlying quasilattice here
is $\Z\times(\Z+a\Z)$, we have, by
Example~\ref{quasidue},
$$\deg_{\textsc{NR}}(\C\P^2_{(1,1,a)})=
\left\{\begin{array}{l}
0 \quad\quad a\in\Q \\
1 \quad\quad a\notin\Q.
\end{array}
\right.
$$
}
\end{example}
\begin{example}[Generalized Hirzebruch surface]\label{quasihirze}
\rm{Take now the generalized
Hirzebruch surface $\H_a$. This is the toric quasifold corresponding to the triple
$$
(\Sigma, \mathbb{Z}\times (\mathbb{Z}+a\mathbb{Z}), \{(1,0), (-1,-a), (0,1),(0,-1)\}),
$$
where $a$ is again any positive real number.
$\Sigma$ is the
complete simplicial fan in $\R^{2}$ whose rays are generated by the vectors
$(1,0)$, $(-1,-a)$, $(0,1)$, and $(0,-1)$.
The toric quasifold $\H_a$ was first introduced in both the symplectic and complex setting in 
\cite{hirze}; see \cite{atq} for the algebraic--geometric version. The corresponding quasilattice is
again $\Z\times(\Z+a\Z)$, so we have
$$\deg_{\textsc{NR}}(\H_a)=
\left\{\begin{array}{l}
0 \quad\quad a\in\Q \\
1 \quad\quad a\notin\Q.
\end{array}
\right.
$$
}
\end{example}
\begin{example}[The Penrose kite quasifold]\label{kite}
\rm{Let us consider the toric quasifold
$X_K$ corresponding to the triple
$$
(\Sigma, Q_5, \{v_1, v_2, v_3,v_4\}),
$$
where $v_1=-u_1$, $v_2=-u_3$, $v_3=u_2$, $v_4=u_4$
(see Example~\ref{quasicinque}). 
$\Sigma$ is the complete simplicial fan 
whose rays are generated by these vectors. It
is the normal fan of
the Penrose kite. This toric quasifold
 was introduced and studied from the symplectic viewpoint in
\cite{kite}; see \cite[Example~3.2]{laurent} for the complex case and \cite{atq} for the algebraic--geometric. 
Since the underlying
quasilattice is the pentagonal quasilattice $Q_5$, by Example~\ref{quasicinque}, we have
that $\deg_{\textsc{NR}}(X_K)=2$. The toric quasifold $X_K$ is totally nonrational.}
\end{example}

\section{Nonrationality degree of fans and polytopes}\label{fan}
Let $\Sigma$ be a complete simplicial fan  in $\R^n$.
Take the set
$\mathcal{Q}_\Sigma$ of all quasilattices $Q\subset \R^n$ that have nonzero
intersection with all rays of $\Sigma$.
\begin{dfn}[Nonrationality degree of a fan]
We define the nonrationality degree of 
$\Sigma$ to be the nonnegative integer $$\deg_{\textsc{NR}}(\Sigma)=\min_{Q\in\mathcal{Q}_\Sigma}\deg_{\textsc{NR}}(Q).
$$
\end{dfn}
Note that the minimum is taken over a nonempty set of nonnegative integers; hence, it is attained.

Let us consider now an $n$--dimensional simple polytope $\Delta\subset \R^n$
(a standard reference for polytopes is \cite{ziegler}).
Let $\Sigma_\Delta$ be its normal fan; it is a complete simplicial fan. We recall that the polytope is rational if and only if its
normal fan is.
\begin{dfn}[Nonrationality degree of a polytope]
\rm{We define the 
\textit{nonrationality degree} of
$\Delta$ to be the nonrationality degree of 
$\Sigma_\Delta$.}
\end{dfn}
\begin{rem}\rm{
Clearly, the fan $\Sigma$ is rational if and only if its nonrationality degree is $0$. Notice also that $$\deg_{\textsc{NR}}(\Sigma)
\leq \deg_{\textsc{NR}}(X_\Sigma)$$ 
for any associated quasifold $X_\Sigma$. In other words, the nonrationality degree of the fan provides a lower bound for the nonrationality degree of the toric quasifold. A similar discussion applies to the polytope $\Delta$.}
\end{rem}
Let us now compute the nonrationality degree of the fans that we introduced in the examples of Section~\ref{quasifold}.
\begin{example}
\rm{The fan of Example~\ref{quasisphere}
is rational with respect to $\Z$, so its nonrationality degree is $0$.}
\end{example}

\begin{example}
\rm{The fans of Examples~\ref{quasiprojective} and \ref{quasihirze} are rational with respect to the lattice $\Z\times a\Z$, so the nonrationality degree of both is also $0$.}
\end{example}\begin{example}
\rm{The Penrose kite (see Example~\ref{kite}), and its normal fan,
are nonrational. In fact, fix arbitrary positive real multiples of the 
four fifth roots of unity that generate the rays of $\Sigma$: $\{r_1u_1,r_2u_2,r_3u_3,r_4u_4\}$
(see Example~\ref{quasicinque}). We claim that no lattice contains all these vectors. Suppose that $L$ is such a lattice, and let 
 $\{w_1,w_2\}$ be a basis of $L$. Then, for all $i$, $r_iu_i=m_iw_1+n_iw_2$, with $m_i,n_i\in\Z$.  Denote by $a_{ij}=\det(r_iu_i,r_ju_j)$. We have $$a_{ij}=r_ir_j\det(u_i,u_j)=r_ir_j\sin\left(\frac{2(j-i)\pi}{5}\right).$$
However, we also have
 $$a_{ij}=(m_in_j-m_jn_i)\det(w_1,w_2).$$ 
 Consider now the ratio
 $$\frac{a_{12}a_{34}}{a_{41}a_{23}}$$ 
 By the first identity, it is 
 equal to the golden ratio $\frac{1+\sqrt{5}}{2}$, whereas the second identity shows that it is rational, a contradiction.
The nonrationality degree is actually $1$: normalize the fifth roots of unity so that each has $x$--coordinate equal to either 0 or 1, and then consider their $\Z$--span, $\widetilde{Q}_5$. 
By construction, $\deg_{\textsc{NR}}(\widetilde{Q}_5)=1$.}
\end{example}

\section{Gordan's Lemma revisited}\label{gordan}
Let $Q$ be a quasilattice in $\R^n$
and consider a maximal simplicial cone $\sigma$ in $\R^n$ whose rays have nonzero intersection with $Q$. Gordan's Lemma (see, for example, \cite[Proposition~1.2.17]{cox}) fails in this setting.
 As it turns out, this failure is measured precisely by the nonrationality degree of $Q$:
\begin{thm}[Gordan's Lemma revisited]\label{gordanlemma} Let $Q$ be a quasilattice in $\R^n$ and let $\sigma$ be a maximal simplicial cone whose rays all have nonzero intersection with $Q$. Then the semigroup $\sigma\cap Q$ is finitely generated if and only if $\deg_{\textsc{NR}}(Q)=0$.
\end{thm}
\noindent\begin{preuve} If $\deg_{\textsc{NR}}(Q)=0$, then $Q$ is a lattice, and the semigroup $\sigma\cap Q$ is finitely generated by Gordan's Lemma.
Conversely, suppose that the semigroup $\sigma\cap Q$ is finitely generated. We first prove that it is discrete.
First, consider a finite set of generators of $\sigma\cap Q$. Since each ray of $\sigma$ has nonzero intersection with $Q$, the above set of generators
must contain at least one generator lying on each ray of $\sigma$. Indeed, in a simplicial cone, a ray cannot be generated unless one of the generators lies on that ray. Denote these $n$ vectors by $\{v_1,\ldots,v_n\}$ and let $\{v_1,\ldots,v_n,v_{n+1},\ldots,v_{n+k}\}$ be the whole set of generators of the semigroup $\sigma\cap Q$. Then, for each $j=n+1,\ldots,n+k$ we have $v_j=\sum_{l=1}^n a_{lj}v_l$, with $a_{lj}$ nonnegative real numbers. Write each $a_{lj}=m_{lj}+b_{lj}$, with $m_{lj}$ a nonnegative integer and $0\leq b_{lj}<1$. Let $u_j=\sum_{l=1}^n b_{lj}v_l$ and discard the vectors $u_j$ that are equal to $0$. After renumbering, we are left with $u_{n+1},\ldots,u_{n+h}$, with $h\leq k$.  Then it is straightforward to verify that the set
$\{v_1,\ldots,v_n,u_{n+1},\ldots,u_{n+h}\}$ still generates the semigroup $\sigma\cap Q$. Moreover, the vectors of this set all lie in the zonotope
$$\Delta=\left\{\sum_{i=1}^n t_iv_i\;|\; 0\leq t_i \leq 1\right\}$$ 
defined by $v_1,\ldots,v_n$. We now prove that $\Delta\cap Q=\Delta\cap(Q\cap \sigma)$ is finite.
Indeed, let $u\in Q\cap \Delta$. Then, since $u\in Q\cap\sigma$, we have $$u=\sum_{i=1}^nm_iv_i+\sum_{j=n+1}^{n+h} n_{j}u_{j}$$ with $m_i,n_j$ nonnegative integers. Hence
$$u=\sum_{i=1}^n\left(m_i+\sum_{j=n+1}^{n+h}n_j b_{ij}\right)v_i.$$ The matrix $(b_{ij})$ has entries $0\leq b_{lj}<1$ and all of its columns are nonzero, since the $u_j$'s are nonzero and the $v_i$'s are linearly independent. On the other hand, since $u\in Q\cap\Delta$ we must have $$0\leq m_i+\sum_{j=n+1}^{n+h}n_j b_{ij}\leq1.$$
Each $m_i$ can only be $0$ or $1$. Moreover, since every column of the matrix $(b_{ij})$ is nonzero, for every $j=n+1,\ldots,n+h$ there exists an index $i$ such that $b_{ij}>0$. Therefore,
$
n_jb_{ij}\leq \sum_{l=n+1}^{n+h} n_l b_{il}\leq 1,
$
which implies that
$
n_j\leq \frac{1}{b_{ij}}.
$
Hence each $n_j$ is bounded above and there are only finitely many possible tuples
$
(m_1,\ldots,m_n,n_{n+1},\ldots,n_{n+h}).
$ Therefore
 $\Delta\cap(Q\cap \sigma)$ is finite.
Moreover, the semigroup is the union of the translates $\Delta+\sum_{i=1}^n m_i v_i$ with $m_i$ nonnegative integers. It is therefore discrete.

Finally, we observe that this implies that $Q$ is in fact a lattice. Indeed, if it were a quasilattice of rank greater than $n$, then, by Corollary~\ref{grado quasireticolo}, $\deg_{\textsc{NR}}(Q)>0$. Hence, the vector subspace $V$ such that $Q\cap V$ is dense in $V$ would have positive dimension. Consequently, there would exist a sequence $(q_m)\subset Q\setminus\{0\} $ converging to $0$. Now, consider the point $v=v_1+\cdots+v_n$ in the interior of $\sigma$. Then, the sequence $(q_m+v)$ is contained in the interior of $\sigma$ for all sufficiently large $m$. Hence $(q_m+v)$  is a sequence in $Q\cap\sigma$ converging to $v$, contradicting the discreteness of $Q\cap\sigma$. 
\end{preuve} 

\begin{rem}\label{ring}
\rm{We have already seen in Remark~\ref{hom} that, if $\deg_{\textsc{NR}}(Q)\neq 0$, the $\Z$--module $\operatorname{Hom}(Q,\Z)$ can no longer be thought of as a quasilattice in $(\R^n)^*$. Even if that were the case, by Theorem~\ref{gordanlemma},
the semigroup algebra associated with the cone
$\sigma$ would not be finitely generated. Consequently, it cannot, in general, serve as the coordinate ring of an affine toric quasifold.}
\end{rem}
\section*{Declarations}
\subsection*{Funding}
Both authors are partially supported by MUR (Italy), PRIN Project 2022AP8HZ9 ``Real and Complex Manifolds: Topology, Geometry and Holomorphic Dynamics", and by GNSAGA, INdAM (Italy).
\subsection*{Conflicts of interest}
The authors declare that they have no conflicts of interest.
\subsection*{Data availability}
No new data were created or analyzed in this study. Data sharing is not applicable to this article.

\bigskip

{\small 

\noindent
Dipartimento di Matematica e Informatica "U. Dini"\\ Universit\`a degli Studi di Firenze \\
Viale Morgagni 67/A, 50134 Firenze, ITALY

\noindent
fiammetta.battaglia@unifi.it, elisa.prato@unifi.it}


\begin{thebibliography}{AA}

\bibitem{cx} F. Battaglia, E. Prato, Generalized toric varieties for simple nonrational convex polytopes, \textit{Int. Math. Res. Not.} \textbf{24} (2001), 1315--1337.

\bibitem{kite} F. Battaglia, E. Prato, The symplectic Penrose kite, \textit{Comm. Math. Phys.} \textbf{299} (2010), 577--601.

\bibitem{laurent} F. Battaglia, E. Prato, Generalized Laurent monomials in nonrational toric geometry, \textit{Contemp. Math.} {\bf 794} (2024), 179--193.

\bibitem{atq} F. Battaglia, E. Prato, Algebraic Toric Quasifolds, arXiv:2604.15192 [math.AG] (2026), 29 p..

\bibitem{hirze} F. Battaglia, E. Prato, D. Zaffran, Hirzebruch surfaces in a one--parameter family,
\textit{Boll. Unione Mat. Ital.} {\bf 12} (2019), 293--305.

\bibitem{bz1} F. Battaglia, D. Zaffran, Foliations modeling nonrational simplicial toric varieties, \textit{Int. Math. Res. Not.} \textbf{2015}, 11785--11815. 

\bibitem{boivin} A. Boivin, Non--simplicial quantum toric varieties, \textit{Adv. Math.} \textbf{441} (2024) 109553, 34 pages.

\bibitem{bosio} F. Bosio, Vari\'et\'es complexes compactes: une g\'en\'eralisation de la construction de Meersseman et de L\'opez de Medrano--Verjovsky,
\textit{Ann. Inst. Fourier} {\bf 51} (2001), 1259--1297.

\bibitem{bourbaki}
N. Bourbaki, \textit{Éléments de mathématique. Livre III: Topologie générale.
Chapitres V--VIII},
Hermann, Paris, 1963.

\bibitem{cox} D. Cox, J. Little, H. Schenck, Toric varieties, Graduate Studies in Mathematics \textbf{124}, American Mathematical Society, 2011.

\bibitem{di} P. Donato, P. Iglesias, Exemples de groupes différentiels: flots irrationnels sur le tore,
\textit{C. R. Acad. Sci. Paris Sér. I Math.} \textbf{301} (1985), 127--130.

\bibitem{H} B. Hoffman, Toric symplectic stacks, \textit{Adv. Math.} \textbf{368} (2020), 43 pages.

\bibitem{HS} B. Hoffman, R. Sjamaar, Stacky Hamiltonian actions and symplectic reduction, \textit{Int. Math. Res. Not.} (2020), 15209--15300.

\bibitem{IZP} P. Iglesias--Zemmour, E. Prato, Quasifolds, Diffeology and Noncommutative Geometry, \textit{J. Noncommut. Geom.} \textbf{15} (2021), 735--759.

\bibitem{IKP} H. Ishida, R. Krutowski, T. Panov, Basic cohomology of canonical holomorphic foliations on complex moment--angle manifolds, {\em Int. Math. Res. Not.} (2022),  5541--5563.

\bibitem{KLMV} L. Katzarkov, E. Lupercio, L. Meersseman, A. Verjovsky, Quantum (non--commutative) toric geometry: Foundations, \textit{Adv. Math.} \textbf{391} (2021), 110 pages.

\bibitem{ldm} S. L\'opez de Medrano and A. Verjovsky, A new family of complex, compact, non symplectic manifolds,
\textit{Bull. Braz. Math. Soc.} {\bf 28} (1997), 253--269.

\bibitem{M} L. Meersseman,
A new geometric construction of compact complex manifolds in any dimension,
\textit{Math. Ann.} {\bf 317} (2000), 79--115.

\bibitem{p} E. Prato, Simple non--rational convex polytopes via symplectic geometry, \textit{Topology} \textbf{40} (2001), 961--975.

\bibitem{ratiu} T. Ratiu, T. N. Zung, Presymplectic convexity and (ir)rational polytopes, \textit{J. Symplectic Geom.} 
\textbf{17} (2019), 1479--1511.

\bibitem{senechal} M. Senechal,
{\it Quasicrystals and geometry},
Cambridge University Press, Cambridge, 1995.

\bibitem{serre}
J.--P. Serre, \textit{Cours d'arithmétique}, 3rd ed.,
Presses Universitaires de France, Paris, 1988.

\bibitem{ziegler} G. Ziegler, Lectures on polytopes, Graduate Texts in Mathematics \textbf{152}, Springer, 1995.

\end{thebibliography}
\end{document}